\documentclass[12pt]{amsart}

\usepackage{amsmath}
\usepackage{amssymb}
\usepackage{mathrsfs}

\newtheorem{thm}{Theorem}
\newtheorem*{semiconj_dichot}{Theorem \ref{semiconj_dichot}}
\newtheorem{cor}{Corollary}
\newtheorem{conj}{Conjecture}
\newtheorem{lem}{Lemma}
\newtheorem{prop}{Proposition}
\newtheorem{claim}{Claim}

\newtheorem{obs}{Observation}

\theoremstyle{remark}
\newtheorem*{rem}{Remark}

\newcommand\PLoI{\mathrm{PL}_+ I}
\newcommand\HomeoI{\mathrm{Homeo}_+ I}
\newcommand\Homeo{\mathrm{Homeo}}
\newcommand\R{\mathbf{R}}
\newcommand\Z{\mathbf{Z}}
\newcommand\Q{\mathbf{Q}}
\newcommand\gen[1]{\langle #1 \rangle}
\newcommand\supt{\operatorname{supt}}

\newcommand\dom{\operatorname{dom}}
\newcommand\mand{\textrm{ and }}

\newcommand\Uscr{\mathscr{U}}
\newcommand\Vscr{\mathscr{V}}
\newcommand\Fscr{\mathscr{F}}
\title[Subgroups of $\PLoI$ not embedding into $F$]
{Subgroups of $\PLoI$ which do not embed\\ into Thompson's group $F$}



\thanks{
The authors would like to thank Collin Bleak, Matthew Brin, and Martin Kassabov
for suggesting comments and corrections to earlier versions of this paper. 
The research of the second author is support by  NSF grant DMS-1854367.
}

\author{James Hyde}
\author{Justin Tatch Moore}

\address{Department of Mathematics \\ Cornell University\\
Ithaca, NY 14853--4201 \\ USA}

\begin{document}

\begin{abstract}
We will give a general criterion --- the existence of an \emph{$F$-obstruction} ---
for showing that
a subgroup of $\PLoI$ does not embed into Thompson's group $F$.
An immediate consequence is that Cleary's ``golden ratio'' group $F_\tau$
does not embed into $F$.
Our results also yield a new proof that Stein's groups
$F_{p,q}$ \cite{grps_PLoI} do not embed into $F$,
a result first established by Lodha using his theory of coherent actions.
We develop the basic theory of $F$-obstructions and show that they exhibit certain
rigidity phenomena of independent interest.

In the course of establishing the main result of the paper, we prove a dichotomy theorem 
for subgroups of $\PLoI$.
In addition to playing a central role in our proof, it is strong enough to
imply both Rubin's Reconstruction Theorem restricted to the class
of subgroups of $\PLoI$ and also Brin's Ubiquity Theorem.
\end{abstract}

\maketitle

\section{Introduction}

In this article, we aim to give a partial answer to the following question:
\emph{When does a group of piecewise linear homeomorphisms of the unit interval
fail to embed into Richard Thompson's group $F$?}
Thompson's group $F$ is the subgroup of $\PLoI$
consisting of those functions whose breakpoints occur at dyadic rationals and whose slopes are powers of $2$.
We isolate the notion of an \emph{$F$-obstruction} based on Poincar\'e's \emph{rotation number} and
show that subgroups of $\PLoI$ which contain $F$-obstructions do not embed into $F$.

In the course of proving the main result of the paper, we establish a dichotomy theorem for subgroups
of $\PLoI$.
This result seems likely to be of independent interest as it is already sufficiently powerful to prove both 
Brin's Ubiquity Theorem \cite{ubiquity} and the restriction of 
Rubin's Reconstruction Theorem \cite[Corollary 3.5(c)]{RubinTheorem} to the class of subgroups of $\PLoI$
(see also  \cite[Theorem 4]{McCleary} and \cite[E16.3]{BieriStrebel} which were precursors to \cite{RubinTheorem}).

\subsection{Rotation numbers and $F$-obstructions}

Recall that if $\gamma$ is a homeomorphism  of the circle $\R/\Z$, then the \emph{rotation number} of
$\gamma$ is defined to be
$$
\lim_{n \to \infty} \frac{\widetilde{\gamma}^n(x) - x}{n}
$$
modulo $1$
where $\widetilde{\gamma}:\R \to \R$ is a lift of $\gamma$
(this limit always exists and, modulo 1, does not depend on $x$ or the choice of $\widetilde{\gamma}$).
Observe that if $\gamma$ is a rotation of $\R/\Z$ by $r \in (0,1)$, then we can take
$\widetilde{\gamma}(x) = x+r$ and the rotation number of $\gamma$ is $r$.
In fact Poincar\'e showed that if $\gamma$ is any homeomorphism such that no finite power has a fixed point,
$\gamma$ is semiconjugate to the irrational rotation specified by its rotation number.
On the other hand, if $\gamma^q$ has a fixed point for some $q$, then we can take $x \in \R$ and $\widetilde{\gamma}$
such that $\widetilde{\gamma}^q(x) = x+p$ for some $p \in \Z$ with $0 \leq p < q$.
It follows that the rotation number is $p/q$.

If $f,g \in \HomeoI$ and $s \in I$ are such that $$s < f(s) \leq g(s) < f(g(s)) = g(f(s))$$
then the \emph{rotation number $f$ modulo $g$ at $s$} is the rotation number of the
function on $[s,g(s))$ defined by $x \mapsto  g^{-m}(f(x))$ 
where $m$ is such that $s \leq g^{-m}(f(x)) < g(s)$.
This map is a homeomorphism of a circle when $[s,g(s))$ is given a suitable topology.
 
A pair $(f,g)$ of elements of $\PLoI$ is an \emph{$F$-obstruction} if there is an $s$ such that either:
\begin{itemize}

\item $s < f(s) \leq g(s) < f(g(s)) = g(f(s))$
and the rotation number of $f$ modulo $g$ at $s$ is irrational;

\item $s > f(s) \geq g(s) > f(g(s)) = g(f(s))$
and the rotation number of $f^{-1}$ modulo $g^{-1}$ at $f(g(s))$ is irrational.

\end{itemize}
It follows from work of Ghys and Sergiescu \cite{T_rational_rot} that the standard way of representing
$F$ in $\PLoI$ does not contain any $F$-obstructions
(see Section \ref{F_obs_in_F:sec}).

The main result of this paper is that 
the property of being an $F$-obstruction is preserved by monomorphisms into $\PLoI$.

\begin{thm} \label{main_result}
If $(f,g)$ is an $F$-obstruction
and $\phi:\gen{f,g} \to \PLoI$ is a monomorphism,
then $(\phi(f),\phi(g))$ is an $F$-obstruction.
In particular, if $G \leq \PLoI$ contains an $F$-obstruction,
then $G$ does not embed into $F$.
\end{thm}

\subsection{A dichotomy for subgroups of $\PLoI$}

Theorem \ref{main_result} is first established for $F$-obstructions which generate a group
with a single \emph{orbital} --- a component of support. 
The general case is then handled by way of a dichotomy theorem for subgroups of $\PLoI$.
This dichotomy is strong enough to imply both Brin's Ubiquity Theorem \cite{ubiquity}
and a form of Rubin's Reconstruction Theorem \cite[Corollary 3.5(c)]{RubinTheorem} for subgroups of $\PLoI$
(see Section \ref{resolvable:sec}).

If $G \leq \PLoI$, then we say that $J$ is a \emph{resolvable} orbital of $G$ if $J$ is an orbital of $G$ and
$\{\supt(g) \cap J \mid g \in G\}$ forms a base for the topology on $J$.
If $G \leq \HomeoI$, a partial function $\psi : I \to I$ is \emph{$G$-equivariant}
if its domain is $G$-invariant and for all $g \in G$ and $x \in \dom(\psi)$, $\psi(g(x)) = g(\psi(x))$.
Our dichotomy theorem can now be stated as follows:

\begin{thm} \label{semiconj_dichot}
Suppose that $G \leq \PLoI$
and $J$ is a resolvable orbital of $G$.
If $K_i$ $(i < n)$ is a sequence of orbitals of $G$,
then exactly one of the following is true:
\begin{enumerate}

\item \label{kernel}There is a $g \in G$ whose support intersects $J$ but is disjoint from $K_i$ for all $i < n$.
\item \label{semiconj}There is an $i < n$ and a monotone surjection
$\psi: K_i \to J$ which is $G$-equivariant.

\end{enumerate}
\end{thm}

\subsection{Corollaries of Theorem \ref{main_result}}
Our original motivation for proving Theorem \ref{main_result} is the following corollary
(see Section \ref{examples:sec} for the definitions of $F_\tau$ and $F_{p,q}$).

\begin{cor} \label{Ftau}
Cleary's group $F_\tau$ does not embed into $F$.
\end{cor}
\noindent
Theorem \ref{main_result} also gives a new proof of the following result first proved by Lodha
using his theory of coherent actions.

\begin{cor} \cite{coherent_actions} \label{Fpq}
Stein's groups $F_{p,q}$ do not embed into $F$ if $p,q$ are relatively prime natural numbers.
\end{cor}

In the next corollary, we view $\PLoI$ as consisting of functions from $\R$ to $\R$ by defining its elements
to be the identity outside of $I$.
Here $F^{t \mapsto t -\xi}$ is the set of conjugates of elements of $F$ by $t \mapsto t -\xi$.

\begin{cor} \label{trans_F}
If $0 < \xi < 1$ is irrational, then $\gen{F \cup F^{t \mapsto t-\xi}}$ does not embed into $F$. 
\end{cor}

In the course of proving Theorem \ref{main_result}, we will also establish the following results.
An $F$-obstruction is \emph{basic} if the group it generates has connected support.

\begin{thm} \label{top_conj}
If two basic $F$-obstructions generate isomorphic groups,
then the groups are topologically conjugate via a
homeomorphism of their supports.
\end{thm}

\begin{thm} \label{F_embed}
If $(f,g)$ is an $F$-obstruction, then $F$ embeds into $\gen{f,g}$.
\end{thm}

\noindent
Theorem \ref{F_embed} generalizes a result of Bleak \cite[\S3.3]{alg_class}
which asserts that if $G \leq \PLoI$ and the left or right
group of germs at some point is nondiscrete, then $F$ embeds into $G$.

We conjecture that the converse to Theorem \ref{main_result} holds for finitely generated groups.

\begin{conj}
If $G \leq \PLoI$ is finitely generated and
does not contain an $F$-obstruction, then $G$ embeds into $F$.
\end{conj}
\noindent
Notice that this conjecture implies every finitely generated subgroup of $\PLoI$
either contains a copy of $F$ or else embeds into $F$;
whether such a dichotomy holds was asked by Matthew Brin.

This paper is organized as follows.
After recalling some terminology and notation in Section \ref{notn:sec} and establishing that $F$ does not contain $F$-obstructions in
Section \ref{F_obs_in_F:sec}, 
we will prove Theorem \ref{F_embed} in Section \ref{F_embed:sec}.
Section \ref{resolvable:sec} proves Theorem \ref{semiconj_dichot}.
Section \ref{main:sec} uses Theorem \ref{semiconj_dichot} to complete the proofs of
Theorems \ref{main_result} and \ref{top_conj}.
It also illustrates how Theorem \ref{semiconj_dichot} can be used to derive
Brin's Ubiquity Theorem and Rubin's Theorem for subgroups of $\PLoI$.
Finally, the computations needed for Corollaries \ref{Ftau}--\ref{trans_F}
are presented in Section \ref{examples:sec}.

\section{Preliminaries}
\label{notn:sec}

Throughout the paper, counting will start at 0 and $i,j,k,l,m,n$ will only be used to denote integers.
If $A$ and $B$ are subsets of an ordered set, we will sometimes write $A \leq B$ to indicate
that every element of $A$ is less than every element of $B$.

As already mentioned, $\HomeoI$ is the collection of orientation preserving homeomorphisms
of the unit interval $I:=[0,1]$.
$\HomeoI$ is a group with the operation of composition.
$\PLoI$ is the subgroup of $\HomeoI$ consisting of those elements which are piecewise linear.
If $f \in \PLoI$, we say that $s$ is a \emph{breakpoint} of $f$ if the derivative of $f$ at $s$ is
undefined.
If $s$ is not a breakpoint of $f$, we will refer to $f'(s)$ as the \emph{slope} of $f$ at $s$.
Thompson's group $F$ consists of those elements of $\PLoI$ whose slopes are integer powers
of $2$ and whose breakpoints are in $\Z[\frac{1}{2}]$.
When there is a need to emphasize that we are working with this particular group and not an isomorphic
copy, we will refer to it as the \emph{standard model of $F$}.
The reader is referred to \cite{CFP} for the basic analysis of Thompson's group $F$ and
\cite{PLoI} for background on $\PLoI$.

Going forward, we will adopt the convention common in the literature that elements of $\HomeoI$ act on $I$ from the right.
Thus we will write $xf$ for the application of $f \in \HomeoI$ to $x \in I$.
If $f \in \HomeoI$, then the \emph{support} of $f$ is defined to be
$$
\supt(f) := \{ x \in I \mid xf \ne x\}.
$$
If $A \subseteq \HomeoI$, then the support of $A$ is defined to be
$$\supt A := \{x \in I \mid \exists g \in A\ (xg \ne x) \} = \bigcup \{\supt(g) \mid g \in A\}.$$
Notice that $\supt A = \supt \gen{A}$.
We will write $\overline{\supt} A$ for the closure of $\supt A$.
A connected component of the support of $f$ is an \emph{orbital} of $f$;
similarly one defines the orbital of a subgroup of $\HomeoI$.
If $f$ has a single orbital, we will say that $f$ is a \emph{bump}.
If $f$ is a bump and $s f > s$ for some (equivalently all) $s$ in its support, then we say $f$
is a \emph{positive bump}; otherwise $f$ is a \emph{negative bump}.

If $f \in \HomeoI$ and $X \subseteq I$ is a union of orbitals and fixed points of $f$,
then $f |_X \in \HomeoI$ is defined by
$$s f |_X := 
\begin{cases}
sf & \textrm{ if } s \in X \\
s & \textrm{ otherwise}.
\end{cases}
$$
This map will be referred to as the \emph{projection to $X$}.
If $G \leq \HomeoI$ and $X$ is a union of orbitals and fixed points of $G$,
then the \emph{projection of $G$ to $X$} is
the image of $G$ under the homomorphism $f \mapsto f |_X$;
we will sometimes use ``the projection of $G$ to $X$'' to refer to the homomorphism itself.

If $f,g$ are elements of a group $G$, define
$g^f := f^{-1} g f$ and
$[f,g] := f^{-1}g^{-1} f g = f^{-1} f^g = (g^{-1})^f g$.
It is easily checked that if $f,g \in \PLoI$, then $\supt(f^g) = \supt(f)g$.
If $A$ and $B$ are sets of group elements, we will write $[A,B]$ for $\{[a,b] \mid a \in A \mand b \in B\}$.
The subgroup of $G$ generated by $[G,G]$ is denoted $G'$.
If $G = G'$, then we say that $G$ is \emph{perfect}.
If $G,H \leq \PLoI$, we will say that \emph{$G$ commutes with $H$} if every element of $G$ commutes
with every element of $H$.

We finish this section with some well known results
which will be needed later in the paper.

\begin{prop} (see \cite{CFP}) \label{fastF}
If $a$ and $b$ are elements of a group such that $[a^b,b^a] = [a^{ba^{-1}},b^{ab^{-1}}]$ is the identity
but $ab\ne ba$, then $\gen{a,b}$ is isomorphic to Thompson's group $F$.\footnote{The
standard presentation of $F$ is $\gen{x_0,x_1 \mid [x_0 x_1^{-1},x_0^{-1} x_1 x_0], [x_0 x_1^{-1},x_0^{-2} x_1 x_0^2] }$.
The presentation stated in Proposition \ref{fastF} is obtained by the substitution $a := x_0 x_1^{-1}$ and $b := x_1^{-1}$.
The proposition follows from this and the fact that the only proper quotients of $F$ are abelian.}
In particular if $s_0 < s_1 < t_0 < t_1$ and $a_0,a_1 \in \HomeoI$ are such that $\supt(a_i) = (s_i,t_i)$ and
$t_0 a_1 \leq s_1 a_0$, then $\gen{a_0,a_1}$ is isomorphic to $F$.
\end{prop}

The next theorem is known as Brin's Ubiquity Theorem.
If $G \leq \PLoI$, $J$ is an orbital of $G$ and $g \in G$, we say $g$ \emph{approaches the left (right) end of $J$}
if the closure of $\supt(g) \cap J$ contains the left (right) endpoint of $J$.

\begin{thm} \label{ubiquity} \cite{ubiquity}
Suppose that $G \leq \PLoI$ and there is an orbital $J$ of $G$ such that some element of $G$
approaches one end of $J$ but not the other. 
Then there is a subgroup of $G$ isomorphic to $F$.
\end{thm}
 
\begin{lem} \label{com_supt} \cite{PLoI}
If $G \leq \PLoI$ and $a \in G'$, then $\overline{\supt}(a) \subseteq \supt G$.
\end{lem}

The next lemma is more or less established in \cite{PLoI}
in the course of showing that nonabelian subgroups of $\PLoI$ contain
infinite rank free abelian groups.
We leave the details to the interested reader.

\begin{lem} \label{lem:mov}
If $G$ is a subgroup of $\HomeoI$ and
$X \subseteq \supt G$ is compact, then there exists
$g \in G$ such that for all $k \in \Z$, $X g^k \cap X = \emptyset$. 
\end{lem}

%
%

\section{$F$ doesn't contain $F$ obstructions}

\label{F_obs_in_F:sec}

In this section we'll prove the following proposition.

\begin{prop}
No pair of elements of the standard model of $F$ is an $F$-obstruction.
\end{prop} 

\begin{proof}
Recall that Thompson's group $T$ consists of all piecewise linear homeomorphisms
of $\R/\Z$ which map $0$ to a dyadic rational, whose breakpoints are dyadic rationals,
and whose slopes are powers of $2$.
Ghys and Sergiescu \cite{T_rational_rot} have shown that every element of $T$ has a rational rotation number.
It therefore suffices to show that if $f,g \in F$ and $s \in I$ with $s < sf \leq sg < sfg=sgf$,
then the associated homeomorphism $\gamma$ defined in the introduction is topologically conjugate to
an element of $T$.

Let $s$, $f$ and $g$ be given as above and let $s_0 < s$ be a dyadic rational such that
$s < s_0 g$, noting that $sg < s_0 g^2$.
By conjugating by an element of $F$ and revising $f$, $g$, $s$, and $s_0$ if necessary, we may assume that
for some $k$, $s_0 g = s_0 + 2^{-k}$ and $s_0 g < 1-2^{-k}$.
By further conjugating by an element $h$ of $F$ which satisfies 
$th = t$ if $t \leq s_0 + 2^{-k}$ and $th = t g^{-1} + 2^{-k}$ if $s_0 + 2^{-k} \leq t \leq s_0 g^2$, we may additionally assume
that if $s_0 \leq t \leq s_0 g$, then $t g = t+2^{-k}$. 
(This conjugacy argument is essentially the \emph{staircase algorithm} of \cite{KassabovMatucci}.)
Repeating this procedure on the interval $[s_0 g, s_0 g^2]$,
we may assume without loss of generality that if $s_0 \leq t < s_0 g^2$, then $tg = t+2^{-k}$.

The homeomorphism $\gamma$ associated to this revised choice of 
$f$, $g$ and $s$ is topologically conjugate to the homeomorphism associated to the original choice of $f$, $g$,
and $s$. 
Moreover $\gamma$ is a homeomorphism of $\R/2^{-k}\Z$ which maps dyadic rationals to
dyadic rationals, whose breakpoints are dyadic rationals, and whose slopes are powers of $2$.
Clearly $\gamma$ is topologically conjugate to an element of $T$ and hence by \cite{T_rational_rot}, $\gamma$
has a rational rotation number.
\end{proof}
 
\section{$F$-obstructions yield copies of $F$}

\label{F_embed:sec}

A key step in proving Theorem \ref{main_result} is to demonstrate that
if $f,g \in \PLoI$ is a basic $F$-obstruction and $J := \supt \gen{f,g}$,
then $J$ is a resolvable orbital of $\gen{f,g}$.
When combined with Proposition \ref{fastF}, this readily yields many copies
of $F$ inside $\gen{f,g}$.
The first step is the following lemma.

\begin{lem} \label{notcommute} 
If $(f,g)$ is an $F$-obstruction, then $f$ and $g$ don't commute.
\end{lem}

\begin{proof}
Let $s$ witness that $(f,g)$ is an $F$-obstruction and let $J$ be the orbital of 
$\gen{f,g}$ such that $s \in J$.
If $f |_J$ and $g|_J$ commute, then by \cite{pres_cong_PLoI} there must be an $h$ such
that $f |_J = h^p$ and $g |_J = h^q$ for integers $p$ and $q$.
This implies that the rotation number of $f$ modulo $g$ at $s$ is $p/q \in \Q$,
which is a contradiction.
\end{proof}

For the duration of this section, fix
a basic $F$-obstruction $(f,g)$ and fix an $s \in I$ which witnesses this.
Specifically, set $C:=[s,sg)$ and let $\gamma :C \to C$ be defined by
$$
x \gamma := 
\begin{cases}
xf & \textrm{ if } xf < sg \\
xfg^{-1} & \textrm{ if } sg \leq xf. \\
\end{cases}
$$
Notice that if $s \leq x < sg$ and $sg \leq xf$, then $s \leq xfg^{-1} < sg$;
this last inequality holds since $xf < sgf = sfg \leq sg^2$ by our hypothesis.
Define a metric $d$ on $C$ by
$$
d(x,y):=\min (y-x,sg-y + x-s)
$$
whenever $x < y$.

With this metric, $C$ is homeomorphic to a circle and $\gamma$ is an orientation preserving homeomorphism of $C$.
Our hypothesis is that the rotation number of $\gamma$ is irrational.
Notice that this implies that $sf \ne sg$ (otherwise this would give a rotation number of $0$) and
hence $sf < sg$.

Since $\gamma$ is piecewise linear, MacKay's variation of Denjoy's Theorem \cite{MacKay} implies
the orbits of $\gamma$ are dense and moreover that $\gamma = \alpha^{-1} \theta \alpha$ for
some irrational rotation $\theta$ of $C$ and some homeomorphism $\alpha$ of $C$.
Since $\alpha$ is uniformly continuous, $\theta$ is an isometry, and
$\gamma^n = \alpha^{-1} \theta^n \alpha$, we can witness the uniform continuity of $\gamma^n$ independently of $n$:
every $\epsilon > 0$ there is a $\delta > 0$
such that for all $x,y \in C$ and $n \in \Z$
if $d(x,y) < \delta$, then $d(x \gamma^n , y \gamma^n) < \epsilon$.
Noting that this assertion remains unchanged if we swap the roles of $(x,y)$ and $(x \gamma^n,y \gamma^n)$,
we will sometimes employ the contrapositive of this implication:
if $d(x,y) \geq \epsilon$ and $n \in \Z$, then $d(x\gamma^n , y \gamma^n) \geq \delta$.
Notice that $d(x,y) \leq |x-y|$ and $d(x,y) = |x-y|$ if $|x-y| \leq (sg-s)/2$.
Since $fgf^{-1}g^{-1} \in \PLoI$ and $sfg = sgf$, there are $t > s$ and $c > 0$ such that $xfgf^{-1}g^{-1} = cx + (1-c)s$ whenever $s \leq x \leq t$. 
If $c \leq 1$, then $xfg \leq xgf$ for all $x \in [s,t)$ and if $c \geq 1$, then $xfg \geq xgf$ for all $x \in [s,t)$.

\begin{lem} \label{remain_long}
There is a $\delta > 0$ such that for all $n \geq 0$ and all $x < y$ in $C$ with $|x-y| < \delta$:
\begin{itemize}

\item if $c \leq 1$ and $x \gamma^n < y \gamma^n$, then
there is an $h \in \gen{f,g}$ such that $xh = x \gamma^n < y \gamma^n \leq yh$;

\item if $c \geq 1$ and $x \gamma^{-n} < y \gamma^{-n}$, then 
there is an $h \in \gen{f,g}$ such that $xh = x \gamma^{-n} < y \gamma^{-n} \leq yh$.

\end{itemize}
\end{lem}

\begin{proof}
First observe that $sf < sfg f^{-1} = sgff^{-1} = sg$ and hence if $s \leq x^* < sf$, then $x^* gf^{-1} < sfgf^{-1} = sg$.

\begin{claim} \label{cyclic_claim}
There is a $\delta > 0$ such that for all $s \leq x < y < sg$ with $|x - y| < \delta$ and for all $n \in \Z$:
\begin{itemize}

\item
if $x \gamma^n < y \gamma^n$, then $d(x \gamma^n,y \gamma^n) = |x \gamma^n - y \gamma^n|$;

\item
if $x \gamma^n > y \gamma^n$, then $d(x \gamma^n,y \gamma^n) = |x \gamma^n - s + sg - y \gamma^n|$.

\end{itemize}
\end{claim}

\begin{proof}
Let $\delta > 0$ be such that $\delta < (sg - s)/6$ and for all $s \leq x,y < sg$ and $n \in \Z$
if $d(x,y) < \delta$, then $d(x \gamma^n,y \gamma^n) < (sg-s)/6$.
This implies that whenever $s \leq x,y < sg$ and $n \in \Z$ if $d(x,y) \geq (sg-s)/6$, then
$d(x \gamma^n,y \gamma^n) \geq \delta$.
Now suppose that $s \leq x < y < sg$ are given with $|x-y| < \delta$.
Let $s \leq z < sg$ be such that $\min(d(y,z),d(z,x)) \geq (sg-s)/6$ --- for instance we can take $z$ to be
the midpoint of the longest arc of $C$ connecting $x$ and $y$.
Such a $z$ cannot be between $x$ and $y$ in the cyclic order.

Suppose that $n \in \Z$ and $x \gamma^n < y \gamma^n$.
Since $\gamma^n$ preserves the cyclic order on $C$,
either $z \gamma^n < x \gamma^n$ or $y \gamma^n < z \gamma^n$.
Since $$\min(d(x\gamma^n,z\gamma^n),d(y \gamma^n,z \gamma^n)) \geq \delta,$$ 
it follows that $\delta \leq (x \gamma^n - s) + (sg - y \gamma^n)$.
Since $$d(x\gamma^n,y \gamma^n) = \min (|x \gamma^n - y \gamma^n|, |x \gamma^n - s + sg - y \gamma^n|) < \delta$$
it must be that $d(x\gamma^n,y \gamma^n) = |x \gamma^n - y \gamma^n|$.

On the other hand,
if $n \in \Z$ is such that $x \gamma^n > y \gamma^n$, then
since $\gamma^n$ preserves the cyclic order on $C$, it must be that
$y \gamma^n < z \gamma^n < x \gamma^n$.
Thus
\begin{align*}
\delta \leq \min(d(y\gamma^n, z \gamma^n),d(x \gamma^n,z \gamma^n)) & \leq
\min(|y \gamma^n- z \gamma^n|,|x \gamma^n-z \gamma^n|) \\
& \leq |x\gamma^n-y \gamma^n|
\end{align*}
and therefore $d(x \gamma^n,y \gamma^n) = |x \gamma^n - s + sg - y \gamma^n|$.
\end{proof}

Let $\epsilon > 0$ be such that $\epsilon < sg-sf$ and if $|x^* - y^*| < \epsilon$ and
$x^* f < sg \leq y^* f$, then $y^* fg^{-1} < t$.
Find a $\delta > 0$ satisfying the conclusion of Claim \ref{cyclic_claim}
and such that additionally if $d(x^*,y^*) < \delta$, then $d(x^* \gamma^n,y^* \gamma^n) < \epsilon$ for all $n \in \Z$.

We will now verify the conclusion of the lemma by induction on $n \geq 0$
under the assumption $c \leq 1$;
the case $c \geq 1$ is handled by an analogous computation.
If $n = 0$, then we can take $h$ to be the identity and there is nothing to show.
Now suppose that $n > 0$, $x < y$ and $x \gamma^n < y \gamma^n$.
If $sf \leq x \gamma^n$, then $x \gamma^n = x \gamma^{n-1} f$ and $y \gamma^n = y \gamma^{n-1} f$.
By our induction hypothesis, there is an $h_0 \in \gen{f,g}$ such that
$x \gamma^{n-1} = xh_0$ and $y \gamma^{n-1} \leq yh_0$.
Since $f$ is order preserving, $y \gamma^{n-1} f \leq yh_0 f$ and since $x \gamma^n = xh_0 f$,
$h:=h_0f$ satisfies the conclusion of the lemma.
Similarly, if $y \gamma^n < sf$, then $x \gamma^n = x \gamma^{n-1} fg^{-1}$ and $y \gamma^n = y \gamma^{n-1} fg^{-1}$
and we can apply our induction hypothesis to find an $h_0 \in \gen{f,g}$ such that
$xh_0 = x\gamma^{n-1}$ and $y \gamma^{n-1} \leq yh_0$.
It follows that $h:=h_0fg^{-1}$ satisfies
$$xh = xh_0fg^{-1} = x \gamma^n < y \gamma^n = y \gamma^{n-1} fg^{-1} \leq yh_0 fg^{-1} = yh.$$

Finally, suppose that $x \gamma^n < sf \leq y \gamma^n$.
By choice of $\delta$ and its property asserted in Claim \ref{cyclic_claim}, this implies
that $d(x \gamma^n, y \gamma^n) = |x \gamma^n - y \gamma^n| < \epsilon$.
It follows that $x \gamma^n = x \gamma^{n-1} fg^{-1}$ and $y \gamma^n = y \gamma^{n-1} f$.
Observe that $x \gamma^{n-1} > y \gamma^{n-1}$ and hence $n > 1$ and
$d(x \gamma^{n-1},y \gamma^{n-1}) = |x \gamma^{n-1} - s + sg - y \gamma^{n-1}|$.
Since $d(x \gamma^{n-1},y \gamma^{n-1}) < \epsilon$, 
it follows that $sf < sg-\epsilon < x \gamma^{n-1}$.
Thus $x \gamma^n = x \gamma^{n-2} f^2 g^{-1}$ and $y \gamma^n  = y \gamma^{n-2} fg^{-1} f$.
Observe that $x \gamma^{n-2} < sgf^{-1} \leq y \gamma^{n-2}$.
By the induction hypothesis, there is an
$h_0 \in \gen{f,g}$ such that $xh_0 = x\gamma^{n-2}$ and $y \gamma^{n-2} \leq yh_0$.
Define $h=h_0 f^2 g^{-1}$.
Since $d(x \gamma^{n-2},y \gamma^{n-2}) < \epsilon$ by our choice of $\delta$ and
since $x \gamma^{n-2} f < sg \leq y \gamma^{n-2} f$, 
it follows from our choice of $\epsilon$ that $s \leq y \gamma^{n-2} f g^{-1} < t$.
Therefore 
$$y \gamma^{n-2} fg^{-1} fg \leq y \gamma^{n-2} f g^{-1} gf = y \gamma^{n-2} f^2$$ 
Acting on the right by $g^{-1}$ yields that 
$$y \gamma^n = y \gamma^{n-2} f g^{-1} f \leq y \gamma^{n-2} f^2 g^{-1} \leq yh_0 f^2 g^{-1} = yh$$
and hence $h$ satisfies the conclusion of the lemma.
\end{proof}

\begin{prop} \label{dense_bumps}
Suppose that $(f,g)$ is a basic $F$-obstruction.
There are dense sets $A,B \subseteq J := \supt \gen{f,g}$ such that
if $a \in A$ and $b \in B$ with $a < b$ then there is an $h \in \gen{f,g}$ such
that $\supt(h) = (a,b)$.
In particular, the support of $\gen{f,g}$ is a resolvable orbital of $\gen{f,g}$.
\end{prop}

\begin{proof}
We will first show that there is an $A_0 \subseteq [s,sg)$ which is dense in $[s,sg)$
such that if $a \in A_0$, then $(a,sg]$ is an initial part of the support of some element of $\gen{f,g}$. 
By Lemma \ref{notcommute}, the commutator $[f,g]$ is not the identity and by Lemma \ref{com_supt},
the infimum of its support is $J$.
Let $h_0 \in \gen{f,g}$ be such that $p:= \inf \supt ([f,g]^{h_0})$ is in $(s,sg)$. 
Such an $h_0$ exists since every $\gen{f,g}$-orbit of a point in $J$ intersects $[s,sg)$.
Let $q > p$ be such that 
$(p,q]$ is an initial part of the support of $[f,g]^{h_0}$. 
Let $\delta > 0$ be satisfy the conclusions of Lemma \ref{remain_long} and Claim \ref{cyclic_claim} and moreover satisfy
for all $x < y$ in $C$:
\begin{itemize}

\item if $d(x,y) \geq q-p$, then for all $n$, $d(x\gamma^n,y \gamma^n) \geq \delta$;

\item if $d(x,y) < \delta$ and $x \gamma^n > y \gamma^n$, then $x \gamma^{n-1} < y \gamma^{n-1}$.

\end{itemize}
There are now two cases depending on whether $c \leq 1$.

If $c \leq 1$, then define $A_0 := \{p \gamma^n \mid n \geq 0\}$.
By MacKay's variation of Denjoy's Theorem \cite{MacKay}, $A_0$ is dense in $[s,sg)$.

\begin{claim}
If $a \in A_0$, then either $(a,a+\delta]$ or $(a,sg]$ is an initial part of
the support of some element of $\gen{f,g}$.
\end{claim}

\begin{proof}
If $a \in A_0$, let $n \geq 0$ be such that $a = p \gamma^n$.
If $a < q \gamma^n$, then by choice of $\delta$ 
there is an $h \in \gen{f,g}$ such
that
$$a=ph = p \gamma^n < a + \delta \leq q \gamma^n \leq q h.$$
Thus, $(a,a+\delta]$ is an initial segment of the support of $[f,g]^{h_0h}$. 
If $p \gamma^n > q \gamma^n$, then by choice of $\delta$ we have that 
$af^{-1} = p \gamma^{n-1} <  q \gamma^{n-1}$ and
there is an $h$ such that $p\gamma^{n-1} = ph$ and $q \gamma^{n-1} \leq q h$.
It follows that $(af^{-1}, q \gamma^{n-1}] \subseteq (ph,qh]$ is an initial part of the support of $[f,g]^{h_0 h}$ and 
hence $(a,sg]$ is an initial part of the support of $[f,g]^{h_0 hf}$. 
\end{proof}

Now suppose that $a \in A_0$ and use the density of $A_0$ to
select a sequence $a_0 = a < a_1 < \ldots < a_k = sg$ such
that if $i < k$ then $a_i \in A$ and $a_{i+2} - a_i < \delta$.
For each $i < k$, let $h_i \in \gen{f,g}$ be such that $(a_i,a_{i+1}]$ is an initial part of $\supt(h_i)$ and
$a_{i+2} \leq a_{i+1} h_i$ if $0 \leq i \leq k-2$.
If $h:= \prod_{i < k} h_i$, then $(a,a_1] \subseteq \supt(h)$ and $sg \leq a_1 h$.
It follows $h$ has $(a,sg]$ as an initial part of its support.

If $tfg \geq tgf$, then we define $A_0 := \{p \gamma^n \mid n \leq 0\}$ and
an analogous argument gives the desired conclusion.
Next, using a similar argument construct an analogous dense $B_0 \subseteq [s,sg)$ such that if $b \in B_0$,
then there is an element
of $\gen{f,g}$ whose support has $[s,b)$ as a final segment.
If $a \in A_0$ and $b \in B_0$ with $a < b$, then let $h_0$ and $h_1$ be such that
$(a,sg]$ is an initial segment of the support of $h_0$ and $[s,b)$ is a final segment of the support of $h_1$.
Furthermore select $h_0$ and $h_1$ such that $$a < bh_0 < a h_1 < b$$
and set $h:= [h_0, h_1]$. 
The following are readily checked:
\begin{itemize}

\item $\supt(h) \subseteq (a,b)$;

\item $h \restriction (a,bh_0] = h_1 \restriction (a,bh_0]$
and  $h \restriction [ah_1,b) = h_0 \restriction [ah_1,b)$;

\end{itemize}
In particular, $h$ has no fixed points in $(a, bh_0]$ or in $[ a h_1 , b)$.
Furthermore, $bh_0 \cdot h = b h_1^{-1} h_0 h_1 = b h_0 h_1 > a h_1$ since $a < b h_0$.
It follows that $h$ has no fixed points in $(a,b)$ and hence $\supt(h) = (a,b)$.

Finally, define $A := A_0 \gen{f,g}$ and $B:= B_0 \gen{f,g}$ and observe that $A$ and $B$ are both dense in $J$.
Let $(x,y)$ be a maximal open interval containing $(s,sg)$ such that if $a \in A \cap (x,y)$ and $b \in B \cap (x,y)$
with $a < b$, then there is an element of $\gen{f,g}$ with support $(a,b)$.
It suffices to show that $(x,y) = J$.
Suppose for contradiction that this is not true --- then either $x$ or $y$ are in $J$.

If $x \in J$, let $h \in \{f^{\pm 1},g^{\pm 1}\}$ be such that $xh \in (x,y)$;
such an $h$ exists since $(s,sg) \subseteq (x,y)$.
Notice that $x' := xh^{-1} < x$.
Let $x' < a < b < y$ with $a \in A$ and $b \in B$.
It suffices to show that $(a,b)$ is the support of an element of $\gen{f,g}$ as this will
contradict the maximality of $(x,y)$.
If $x < a$, then $(a,b)$ is the support of an element of $\gen{f,g}$ by our choice of $(x,y)$.
Similarly, if $x' < a < b \leq x$, then $x < ah < bh < y$ and there is an $h_0 \in \gen{f,g}$ with support
$(ah,bh)$.
It follows that $h_0^{h^{-1}}$ has support $(a,b)$.
If $x' < a \leq x < b$, then $x < ah \leq xh$.
Let $b' \in B$ be such that $xh < b' < \min(bh,y)$.
Let $h_0 \in \gen{f,g}$ be a positive bump with support $(ah,b')$, noting that $h_0^{h^{-1}}$ has support $(a,b'h^{-1})$. 
Let $a' \in A$ be such that $a' < b' h^{-1} < b$ and let $h_1$ be a positive bump with support $(a',b)$.
Now $h_0^{h^{-1}} h_1$ is a positive bump with support $(a,b)$.
This gives the desired contradiction.
The case $y \in J$ is handled by an analogous argument. 
\end{proof}

Theorem \ref{F_embed} is an immediate consequence of Proposition \ref{dense_bumps} and
Brin's Ubiquity Theorem \cite{ubiquity}.

\section{A dichotomy for subgroups of $\PLoI$}

\label{resolvable:sec}

In this section we will prove Theorem \ref{semiconj_dichot}.

\begin{semiconj_dichot} 
Suppose that $G \leq \PLoI$
and $J$ is a resolvable orbital of $G$.
If $K_i$ $(i < n)$ is a sequence of orbitals of $G$,
then exactly one of the following is true:
\begin{enumerate}

\item \label{kernel}There is a $g \in G$ whose support intersects $J$ but is disjoint from $K_i$ for all $i < n$.
\item \label{semiconj}There is an $i < n$ and a monotone surjection
$\psi: K_i \to J$ which is $G$-equivariant.

\end{enumerate}
\end{semiconj_dichot}

This theorem will be proved through a series of lemmas.
The first gives a criterion for the existence of an equivariant surjection between orbitals
of a subgroup of $\HomeoI$.

\begin{lem} \label{semiconj_crit}
Suppose that $J$ is a resolvable orbital of $G \leq \HomeoI$ and $K$ is an orbital of $G$.
If there are nonempty open intervals $U$ and $V$ such that
\begin{enumerate}
\item $\overline{U} \subseteq J$ and $\overline{V} \subseteq K$ and
\item \label{intersecting} for all $g \in G$, $Ug \cap U \ne \emptyset$ if and only if $Vg \cap V \ne \emptyset$,	
\end{enumerate}
then there is a $G$-equivariant surjection $\psi:K \to J$ which is continuous and
monotone.
\end{lem}

\begin{proof}
Define
$$V^*:=\bigcup \{Vh \mid (h \in G) \mand (Uh \subseteq U)\}$$
and observe that $V^*$ is an open interval containing $V$.

\begin{claim} \label{upgrade}
For all $g \in G$:
\begin{enumerate}

\item \label{disjoint_equiv}
$Ug \cap U \ne \emptyset$
if and only if
$V^* g \cap V^* \ne \emptyset$;

\item \label{V*_bounded}
$\overline{V^*} \subseteq K$;

\item	 \label{containment_equiv}
$Ug \subseteq U$
if and only if
$V^* g \subseteq V^*$;

\item \label{strong_containment}
if $\overline{Ug} \subseteq U$, then
$\overline{V^* g} \subseteq V^*$.

\end{enumerate}
\end{claim}

\begin{proof}
Let $g \in G$.
If $U \cap Ug \ne \emptyset$, then
$\emptyset \ne V \cap V g \subseteq V^* \cap V^* g$.
Next, suppose that $x \in V^* \cap V^* g$ for some $g$ and
let $h_0$ and $h_1$ be such that
$Uh_0 \cup U h_1 \subseteq U$ and $x \in Vh_0 \cap Vh_1 g$.
It follows that $V \cap Vh_1 g h_0^{-1} \ne \emptyset$, which implies
$U \cap Uh_1 g h_0^{-1} \ne \emptyset$ which in turn implies
$U h_0 \cap U h_1 g \subseteq U \cap Ug \ne \emptyset$.
This establishes (\ref{disjoint_equiv}).

Observe that if $\overline{V^*}$ contains an endpoint of $K$,
then for any $g \in G$, $V^* g \cap V^* \ne \emptyset$.
On the other hand since $J$ is a resolvable orbital of $G$ and $\overline{U} \subseteq J$,
there is a $g \in G$ such that $U g \cap U = \emptyset$.
Thus (\ref{V*_bounded}) follows from (\ref{disjoint_equiv}).

We will now prove (\ref{containment_equiv}).
First suppose that $Ug \subseteq U$ for some $g \in G$.
If $y \in V^* g$, let $h$ be such that $Uh \subseteq U$ and $yg^{-1} \in Vh$.
Then $Uhg \subseteq Ug \subseteq U$ and so $y \in Vhg = Vhg \subseteq V^*$.
Suppose now $Ug$ is not contained in $U$.
Since $J$ is a resolvable orbital of $G$, there is an $h \in G$
such that $Uh$ intersects $Ug$ but not $U$.
It follows from (\ref{disjoint_equiv})
that $V^* h$ intersects $V^* g$ but not $V^*$ and
hence that $V^* g$ is not contained in $V^*$.

Finally, suppose that $\overline{Ug} \subseteq U$ for
some $g \in G$.
Since $J$ is a resolvable orbital of $G$, there are $h_0,h_1 \in G$
such that:
\begin{itemize}

\item $U h_0 \cap U h_1 = \emptyset$,

\item $U h_0 \cup U h_1 \subseteq U$,

\item $U h_0$ and $U h_1$ intersect $U g$ but neither are
contained in $U g$.
	
\end{itemize}
It follows from (\ref{disjoint_equiv}) and
(\ref{containment_equiv}) that these same conditions hold
of $V^*$ in place of $U$.
This implies that the endpoints of $V^* g$ are contained in
$V^* h_0 \cup V^* h_1$ and hence that
$\overline{V^* g} \subseteq V^*$.
\end{proof}

By replacing $V$ with $V^*$ if necessary, we may assume that $V$ has the additional properties of
$V^*$ in Claim \ref{upgrade} --- these will be referred to as the \emph{revised hypotheses on $U$ and $V$}.

Define $\psi$ to consist of all pairs
$(x,y) \in K \times J$ such that for all $g \in G$, whenever $y \in Ug$, $x \in \overline{Vg}$.
To see that $\psi$ is a function, suppose $y_0 \ne y_1$ are in $J$.
Since $J$ is a resolvable orbital of $G$, there are $g_i,h_i \in G$
such that $y_i \in Ug_i \subseteq \overline{Ug_i} \subseteq Uh_i$ and $Uh_0 \cap Uh_1 = \emptyset$.
By our revised hypotheses, $\overline{Vg_0} \cap \overline{Vg_1} = \emptyset$.
Hence there is no $x$ such that $(x,y_0)$ and $(x,y_1)$ are in $\psi$.
Since the graph of $\psi$ is closed, $\psi$ is continuous.
It also follows immediately from the definition that $(x,y) \in \psi$ if and only if $(xg,yg) \in \psi$ and hence
$\psi$ is $G$-equivariant.

Next let us say that two intervals are
\emph{linked} if neither is a subset of the other.
Observe that for any $g \in G$,
$U$ and $Ug$ are linked if and only if $V$ and $Vg$
are linked.
Also, a pair of intervals is linked if and only if 
each contains an endpoint of the other.
If $A$ and $B$ are intervals, then we will write $A <_l B$
to mean that the pair $A$, $B$ is linked and the left endpoint of $A$ is less than the left endpoint of $B$. 
Clearly if $A$ and $B$ is a linked pair of intervals, then exactly one of $A <_l B$ or $B <_l A$.

\begin{claim}
Either for all $g \in G$, $U <_l U g$ implies $V <_l V g$
or for all $g \in G$, $U <_l U g$ implies $V g <_l V$.
\end{claim}	

\begin{proof}
Observe that if $U <_l Ug$, then $Ug^{-1} <_l U$.
Hence if the claim is false, there are $g_0,g_1 \in G$ such that
$Ug_0 <_l U <_l U g_1$ and yet $V <_l Vg_0,Vg_1$.
Since $J$ is a resolvable orbital of $G$, there is an $h \in G$ such that $\supt(h) \cap J \subset U$
and $Ug_0 \cap Ug_1 h = \emptyset$.
Since $U h = U$, $Vh = V$ and because the right endpoint $V$ is in $Vg_1$,
it is also in $Vg_1 h$.
In particular this right endpoint is in both $Vg_0$ and $Vg_1 h$ while $Ug_0$ and
$Ug_1 h$ are disjoint, contrary to our hypothesis.
\end{proof}

\begin{claim}
$\psi$ is monotone.
\end{claim}

\begin{proof}
Suppose that for all $g,h \in G$, $U g <_l U h$ implies $V g <_l V h$.
Let $\psi(x_0) = y_0 < y_1 = \psi(x_1)$ and let $g_i \in G$
be such that $y_i \in U g_i$ and $U g_0 \cap U g_1 = \emptyset$.
By resolvability of $G$ on $J$, there is an $h$ such that
$Uh$ links both $U g_0$ and $U g_1$.
In particular,
$U g_0 <_l U h <_l U g_1$, which implies
$V g_0 <_l V h <_l V g_1$.
Since $x_i \in V g_i$ and $V g_0 \cap V g_1 = \emptyset$,
it follows that $x_0 < x_1$.
Similarly, if $U g <_l U h$ always implies $V g <_l V h$,
then $\psi$ is monotone decreasing.
\end{proof}

\begin{claim}
$\psi$ is a surjection from $K$ to $J$.
\end{claim}

\begin{proof}
In order to see that $\psi$ is a surjection,
let $y \in J$ be given.
By assumption,
$$\Fscr := \{\overline{Vg} \mid (g \in G) \mand (y \in Ug)\}$$
is a pairwise intersecting collection of intervals.
By our revised hypotheses on $U$ and $V$ and
by the resolvability of $G$ on $J$, $\Fscr$
contains elements whose closure is contained in $K$.
Thus $\phi^{-1}(y) = \bigcap \Fscr$ is a nonempty interval.

Now suppose that $x \in K$.
Since $\psi$ is a surjection, its domain is nonempty;
since $\psi$ is $G$-equivariant, its domain $X$ contains elements both to the left and right of $x$.
Set $$x_0:= \sup \{s \in X \mid s \leq x\} \qquad \quad x_1 := \inf \{s \in X \mid x \leq s\}.$$
Notice that since $\psi$ is a monotone surjection, $\psi(x_0) = \psi(x_1)$.
Since we have shown $\psi$-preimages of points are intervals, $x \in \dom (\psi)$.
\end{proof}

\end{proof}

Our strategy for proving Theorem \ref{semiconj_dichot} will now be to carefully select a subgroup $H \leq G$ whose support has nice properties which will allow
us to define intervals $U$ and $V$ as in Lemma
\ref{semiconj_crit}.
The first step toward this goal is the following lemma.
Recall that a group $G$ is \emph{perfect} if $G' = G$.

\begin{lem} \label{seed_subgrp}
Suppose that $J$ is a resolvable orbital of $G \leq \PLoI$ and $K$ is an orbital of $G$.
There is an $H \leq G$ such that:
\begin{enumerate}

\item $H$ is perfect;
\item $H$ has finitely many orbitals;
\item $\supt H \cap J$ is a resolvable orbital of $H$ with closure contained in $J$;
\item $\supt H \cap K$ has closure contained in $K$.

\end{enumerate}	
\end{lem}

This lemma will itself be proved though a series of
lemmas.

\begin{lem} \label{lem:com}
If $A,B \leq \HomeoI$ and
for each $a \in A$ we have 
$\overline{\supt}(a) \subseteq \supt B$,
then $A' \leq \gen{[[A,B],[A,B]]}$.
\end{lem}

\begin{proof}
Suppose that $a_0,a_1 \in A$ are arbitrary; it suffices
to show that $[a_0,a_1] \in [[A,B],[A,B]]$. 
Set $$X := \overline{\supt}\{a_0,a_1\}.$$
By Lemma \ref{lem:mov}, there is a $b \in B$ such that
both $Xb$ and $Xb^{-1}$ are disjoint from $X$.
Since $(a_1^{-1})^b$ is supported on $Xb$ and $(a_0^{-1})^{b^{-1}}$ is supported on $X b^{-1}$, these terms
commute with each other and with $a_0$ and $a_1$, which are supported on $X$.
Thus 
$$\big[[b^{-1},a_0],[b,a_1]\big] = \big[(a_0^{-1})^{b^{-1}} a_0, (a_1^{-1})^b a_1\big] = [a_0,a_1]$$ 
is in $[[A,B],[A,B]]$ as desired.
\end{proof}

\begin{lem} \label{lem:per}
Let $G$ be a subgroup of $\PLoI$.
If $\supt G' = \supt G$,
then $\supt G'' = \supt G$ and $G''$ is perfect.
\end{lem}

\begin{proof}
If $a \in G'$, then by Lemma \ref{com_supt} $$\overline{\supt}(a) \subseteq \supt G = \supt G'.$$
Thus we may apply Lemma \ref{lem:com} to $A = B = G'$ and obtain that
$$G'' \leq \gen{[[G',G'],[G',G']]} \leq G'''.$$
Thus $G'' = G'''$ and $G''$ is perfect.
To see that $\supt G'' = \supt G$,
suppose that $x \in \supt G$.
By assumption, there is a $g \in G'$ such that $x \in \supt(g)$.
By Lemmas \ref{com_supt} and \ref{lem:mov}, there is an $f \in G'$ such that the supports of $g$ and $g^f$
are disjoint.
It follows that $xg = x[f,g]$ and therefore that $x$ is in the support of $[f,g] \in G''$. 
\end{proof}

\begin{lem} \label{lem:persubgrp}
If $G \leq \PLoI$ is perfect and $H \leq G$ is a normal subgroup with
$\supt H = \supt G$, then $G = H$.
\end{lem}

\begin{proof}
Since $H$ is a normal subgroup of $G$, $\gen{[[G,H],[G,H]]} \leq H$.
By Lemma \ref{com_supt} $\overline{\supt}(g) \subseteq \supt G = \supt H$ for every $g \in G$.
Applying Lemma \ref{lem:com} to $A = G$ and $B = H$,
we obtain that $G = G' \leq \gen{[[G,H],[G,H]]} \leq H$.
\end{proof}

\begin{lem} \label{lem:sup}
Let $H_0 \leq \PLoI$ and $J$ be an orbital
of $H_0$ such that $H_0|_J$
is both perfect and the normal closure of a single element.
Then there exists a perfect subgroup $H$ of $H_0$ with finitely many orbitals
such that $H|_J = H_0 |_J$.
\end{lem}

\begin{proof}
Since $H_0 |_J$ is perfect $H_0'' |_J = H_0 |_J$.
Fix $h \in H_0''$ with the normal closure of $h|_J$ in $H_0|_J$ equal to $H_0|_J$.
Let $H_1$ be the normal closure of $h$ in $H_0''$ and let $H_2$ be the normal closure of $h$ in $H_1$.
Define $\Uscr$ to be the orbitals of $H_0$ that are also orbitals of $H_2$ and set $U:=\bigcup \Uscr$.
Define $\Vscr$ to be the orbitals of $H_0$ which intersect $\supt H_2$ but are not orbitals of $H_2$ and
set $V:=\bigcup \Vscr$.

\begin{claim} \label{H2_V_small}
$\overline{\supt} H_2|_V$ is contained in $\supt H''_0$.
\end{claim} 

\begin{proof}
We will first argue that if $L$ is an orbital of $H_1 |_V$, then the closure of $L$ is contained
in the support of $H''_0$.
To see this, suppose $L$ is an orbital of $H_1 |_V$.
If $L$ is an orbital of $H''_0$, then Lemma \ref{lem:persubgrp} implies $H''_0 |_L = H_1 |_L$.
Since this in turn implies $H_2 |_L = H_1 |_L = H''_0 |_L$ and that $L$ is in $\Uscr$, this is impossible.
Let $g \in H''_0$ be such that $Lg \not \subseteq L$ .
Since $Lg$ is contained in the support of $H$ and $L$ is an orbital of $H_1$, 
$Lg$ must be disjoint from $L$ and in particular the closure of $L$ is contained in $\supt H''_0$.

Now let $X$ be the closure of the union of the orbitals of $H_1 |_V$ which intersect $\supt (h)$.
We have shown that $X \subseteq \supt H''_0$.
Since $X$ is $H_1$-invariant, $\supt H_2|_V \subseteq X$.
\end{proof}

By Claim \ref{H2_V_small}, we may find $g \in H_0$ such that 
$V \cap \supt H_2 \cap \supt H_2^g = \emptyset$.
Define $H:=\gen{ [H_2^g,H_2] }$, noting that
$\supt H \subseteq U$.
By Lemma \ref{lem:per}, $H''_0|_U$ is perfect.
By Lemma \ref{lem:persubgrp}, $H_1|_U = H''_0 |_U$ and therefore $H_2 |_U = H''_0|_U$.
Since
$H_0''$ is a normal subgroup of $H_0$, we have ${H_0''}^g = H_0''$.
Putting this all together, we obtain
\begin{align*}      
H = H |_U = \gen{ [H_2^g|_U,H_2|_U]} &= \gen{[{H''_0}^g|_U,H''_0|_U]} \\
& = \gen{[H''_0|_U,H''_0|_U]} = H'''_0|_U = H''_0|_U.
\end{align*}
Thus $H=H''_0|_U$ is perfect and has finitely many components of support.
\end{proof}

\begin{lem} \label{projection_lem}
Suppose that $A,B \leq \PLoI$ are perfect groups 
and $N$ is the normal closure of $B$ in $\gen{A \cup B}$.
If $S$ denotes the union of the orbitals of $\gen{A \cup B}$ which are not orbitals of $N$,
then $A|_S$ is contained in $\gen{A \cup B}$.
\end{lem}

\begin{proof}
We will first show that for all $a \in A$ there exists $b \in N$ such that $ab|_S$ is in $\gen{A \cup B}$.
Define $A_0$ to be the set of all $a \in A$ such that there exists $b \in N$ with
$ab|_S$ is in $\gen{A \cup B}$.
Since $A$ is perfect, it suffices to show that
$[A,A] \subseteq A_0$ and that $A_0$ is a group.

Toward showing $[A,A] \subseteq A_0$,
let $f,g \in A$ and set $X$ equal to the closure of $\supt \{f,g\} \setminus S$.
Since $A$ is perfect, Lemma \ref{com_supt} implies $X$ is contained in $\supt N$.
By Lemma \ref{lem:mov} there is an $h \in N$ such that $Xh \cap X = \emptyset$.
Since $f$ and $g^h$ have disjoint supports, $[f,g^h]$ agrees with the identity on $I \setminus S$. 
In particular $[f,g^h]|_S = [f,g^h]$ is in $\gen{A \cup B}$. 
We may rewrite
$[f,g^h]$ as $[f,g] {(h^{-1})}^{(f^g)} h^{f g} (h^{-1})^g h$. 
Since ${(h^{-1})}^{(f^g)} h^{fg} (h^{-1})^g h$ is in $N$ we have shown that $[f,g]$ is in $A_0$. 

It remains to show that $A_0$ is closed under composition and taking inverses.
Let $a_0,a_1 \in A_0$ and fix $b_0,b_1 \in N$ with $a_0 b_0|_S$ and $a_1 b_1|_S$ in $\gen{A \cup B}$.
Since $a_0a_1 b_0^{a_1}b_1|_S= (a_0b_0)|_S (a_1b_1)|_S$ is in $\gen{A \cup B}$ and $b_0^{a_1}b_1$ is in $N$,
it follows that $a_0 a_1 \in A_0$.
Since $(a_0b_0)^{-1} |_S = a_0^{-1}{(b_0^{-1})}^{a_0^{-1}}|_S$ is in $\gen{A \cup B}$
and ${(b_0^{-1})}^{a_0^{-1}}$ is in $N$,
it follows that $a_0^{-1} \in A_0$.
Since $a_0,a_1 \in A_0$ were arbitrary, $A_0$ is closed under multiplication and
taking inverses and hence is a group.
Thus $A_0 = A$.

Next we will show that $N |_S$ is contained in $\gen{A \cup B}$.
Notice that this is sufficient to complete the proof since if $a \in A$, then for some $b \in N$, $(ab) |_S$ is
in $\gen{A \cup B}$ and hence $a |_S = (ab) |_S (b^{-1}) |_S$ is in $\gen{A \cup B}$.
Since $B$ is perfect, it is generated by $[B,B]$ and hence $N$ is the normal closure of
$[B,B]$ in $\gen{A \cup B}$.
Thus it suffices to show that if $b_0,b_1 \in B$, then $[b_0,b_1]|_S$ is in $\gen{A \cup B}$. 
Toward this end, let $b_0,b_1 \in B$ be arbitrary.
Observe that any endpoint of a connected component $U$ of $S$ is a limit point
of $U \setminus \supt N$.
Thus there is a set $X \subseteq S$ which is a finite union of intervals with endpoints
in $S \setminus \supt N$ such that $\supt \{b_0,b_1\}  \cap S \subseteq X$.

We now claim that $X$ is contained in the support of $A$.
If there were an $x \in X$ fixed by every element of $A$, then $x$ is in some component $L$ of the support
of $B$.
In this case, however, the union of the translates of $L$ by elements of $\gen{A \cup B}$ is an orbital
of both $N$ and $\gen{A \cup B}$, contradicting that $L$ is contained in $S$.
Thus it must be that $X$ is contained in the support of $A$.

By Lemma \ref{lem:mov}, there is an $a \in A$ such that $Xa \cap X = \emptyset$.
Thus there is a $g \in N$ such that $h:=(a g) |_S$ is in $\gen{A \cup B}$.
We will be finished once we show that $[[h ,b_0],b_1] =  [b_0,b_1] |_S$.
Define $Y:=S \setminus X$ and $Z:= I \setminus S$ and set $c:=(b_0^{-1})^h$.
Observe that $X$, $Y$, and $Z$ are all invariant under $b_0$, $b_1$ and $c$.
Furthermore $[h,b_0] = c b_0$ agrees with $b_0$ on $X$ and the identity on $Z$.
Since $b_1$ agrees with the identity on $Y$,
we have that $[[h ,b_0],b_1] = [cb_0,b_1]$ coincides with $[b_0,b_1]$ on $X$ and is
the identity elsewhere.
Since $[b_0,b_1]$ is the identity on $Y$,
it follows that $[[h ,b_0],b_1] =  [b_0,b_1] |_S$.
\end{proof}

\begin{proof}[Proof of Lemma \ref{seed_subgrp}.]
Let $s < t$ be in $J$ and set $U:=(s,t)$.
Define $H_0:= \{g \in G \mid \supt (g) \cap J \subseteq U \}''$.
Since $J$ is a resolvable orbital of $G$, $H_0 |_J$ is perfect by Lemma \ref{lem:per}.
It is easily checked that $U$ is a resolvable orbital of $H_0$ and hence Lemma \ref{lem:persubgrp} implies
that $H_0|_J$ is the normal closure of a single element.
By Lemma \ref{lem:sup}, there is an $H \leq H_0$ such that $H$ is perfect, has finitely many orbitals,
and $H |_J = H_0|_J$.

If the closure of $\supt H \cap K$ is contained in $K$, then $H$ satisfies the conclusion of the lemma.
If the closure of $\supt H \cap K$ contains an endpoint of $K$, then let $g \in G$ be such that
$Ug \cap U = \emptyset$.
Let $N$ be the normal closure of $H^g$ in $\gen{H \cup H^g}$ and let $S$ be the union of the orbitals
of $\gen{H \cup H^g}$ which are not also orbitals of $N$.
Observe that $U \subseteq S$ and that the closure of $S \cap K$ does not contain the endpoints of $K$.
Applying Lemma \ref{projection_lem} to $A = H$ and $B= H^g$, the projection $H |_S$
is contained in $\gen{H \cup H^g} \leq G$ and satisfies the conclusion of the lemma.
\end{proof}

\begin{proof}[Proof of Theorem \ref{semiconj_dichot}]
The theorem is proved by induction on $n$
with the bulk of the proof dedicated to the base case $n=1$.
Suppose that $n=1$ and write $K$ for $K_0$.
Applying Lemma \ref{seed_subgrp}, fix a perfect subgroup $H \leq G$ with finitely many orbitals
such that:
\begin{itemize}

\item $U:=\supt H \cap J$ is a nonempty interval with closure contained in $J$;

\item the closure of $\supt H \cap K$ is contained in $K$;

\item $U$ is a resolvable orbital of $H$;

\item the number of orbitals of $H$ is minimized.

\end{itemize}
If $\supt H \cap K = \emptyset$, then the first alternative of the theorem holds.
Suppose now that this is not the case.
Define $V$ to be the leftmost component of $\supt H \cap K$.
By lemma \ref{semiconj_crit},
it suffices to show that for all $g \in G$, $Ug \cap U \ne \emptyset$ if and only if $Vg \cap V \ne \emptyset$.

\begin{claim}
For all $g \in G$, $Ug \cap U \ne \emptyset$ implies $Vg \cap V \ne \emptyset$
\end{claim}

\begin{proof}
Suppose first for contradiction that for some $g \in G$,
$Ug \cap U \ne \emptyset$ but $Vg \cap V = \emptyset$.
By replacing $g$ with $g^{-1}$ if necessary, we may assume that $V$ is disjoint from the support of $H^g$.
Let $N$ denote the normal closure of $H^g$ in $\gen{H \cup H^g}$ and let
$S$ be the union of all orbitals of $\gen{H \cup H^g}$ which are not orbitals of $N$.
Observe that since $U$ is a resolvable orbital of $H$ and $U \cap Ug \ne \emptyset$, it follows
that $U$ is disjoint from $S$.
On the other hand, $V$ is contained in $S$.
Thus applying Lemma \ref{projection_lem} to $A = H$ and $B = H^g$ yields
that $H|_S$ is a perfect subgroup of $G$.
Consequently if $R = \supt H \setminus S$, then $H_0:=H|_R$ is also contained in $G$.
Since $H_0$ is also perfect, satisfies $H_0|_U = H|_U$ and has fewer orbitals than $H$,
we have contradicted our choice of $H$ to minimize the number of its orbitals.
\end{proof}

\begin{claim}
For all $g \in G$, $Ug \cap U = \emptyset$ implies $Vg \cap V = \emptyset$
\end{claim}

\begin{proof}
Suppose first for contradiction that for some $g \in G$,
$Ug \cap U = \emptyset$ but $Vg \cap V \ne \emptyset$.
By replacing $g$ with $g^{-1}$ if necessary, we may assume that $\inf V \leq \inf Vg$. 
As in the previous claim, let $N$ denote the normal closure of $H^g$ in $\gen{H \cup H^g}$ and let
$S$ be the union of all orbitals of $\gen{H \cup H^g}$ which are not orbitals of $N$.
This time $U$ is contained in $S$ and $V$ is disjoint from $S$.
Lemma \ref{projection_lem} implies $H_0:=H|_S$ is a perfect subgroup of $G$.
Since $H_0$ has fewer orbitals than $H$, is perfect, and satisfies $H_0|_J = H|_J$,
we again contradict our choice of $H$.
\end{proof}

Now suppose that that $n$ is given and that the statement
of the theorem is true for $n$.
Let $K_i$ $(i < n+1)$ be orbitals of $G$.
We need to show that if (\ref{semiconj}) fails, then
there is a $g \in G$ whose support intersects $J$
but not $K_i$ for any $i < n+1$.
By our inductive assumption, there is a $g_0 \in G$
whose support intersects $J$ but is disjoint from
$K_i$ for $i < n$.
Additionally here is a $g_1 \in G$ whose support
intersects $J$ but is disjoint from $K_n$.
Since $J$ is a resolvable orbital of $G$, there is an $h$ such that
$g_0^h$ does not commute with $g_1$.
It follows that $g = [g_0^h,g_1]$ is as desired. 
\end{proof}

\section{Consequences of the dichotomy theorem}

\label{main:sec}

In this section we will give proofs of Theorems \ref{main_result} and \ref{top_conj} using
Theorem \ref{semiconj_dichot}.
We will also illustrate the utility of Theorem \ref{semiconj_dichot} by deriving some known results
as corollaries.

\begin{proof}[Proof of Theorem \ref{main_result}.]
Suppose that $(f,g)$ is an $F$-obstruction and $\phi :\gen{f,g} \to \PLoI$ is an embedding.
By rescaling and translating if necessary, we may assume that the supports of $\gen{f,g}$ and
$\gen{\phi(f),\phi(g)}$ are disjoint.
Let $J$ be an orbital of $\gen{f,g}$ such that for some $s \in J$, the rotation number of $f$
 modulo $g$ at $s$ is irrational. 
Let $K_i$ $(i < n)$ list the orbitals of $\gen{\phi(f),\phi(g)}$.
Observe that since $\phi$ is an injection,
the first alternative of Theorem \ref{semiconj_dichot} applied to $G :=\gen{f\phi(f),g\phi(g)}$
can not hold.
Therefore there is an $i < n$ and a continuous $G$-equivariant monotone
surjection $\psi: K_i \to J$; let $t \in K_i$ be such that $\psi(t) = s$.
Since rotation numbers are preserved by semiconjugacy, 
it follows that the rotation number of $\phi(f)$ modulo $\phi(g)$ at $t$ is irrational
and hence that $(\phi(f),\phi(g))$ is an $F$-obstruction.
\end{proof}

We will now recall the statement and context of Rubin's Reconstruction Theorem.
Suppose that $X$ is a locally compact Hausdorff space and $G$ is a group of homeomorphisms of $X$.
The group $G$'s action on $X$ is \emph{locally dense} if $X$ has no isolated points and
whenever $x \in U \subseteq X$
with $U$ open,
$$\{xg \mid (g \in G) \mand (\supt(g) \subseteq U)\}
$$ is somewhere dense.
It is easily checked that if $X \subseteq I$ is an interval, then this is equivalent to $X$ being a resolvable orbital of $G$.
Rubin's Theorem asserts that $X$ and $Y$ are locally compact and if $G \leq \Homeo X$ and $H \leq \Homeo Y$ are such that the actions of $G$ and
$H$ on their underlying spaces are locally dense, then any isomorphism between $G$ and $H$ is induced by
a unique homeomorphism of $X$ and $Y$ (this is Corollary 3.5(c) of \cite{RubinTheorem}).

We will now show how to derive Rubin's theorem when $G$ and $H$ are subgroups of $\PLoI$.
Notice that Theorem \ref{top_conj} is an immediate consequence of this result and Proposition \ref{dense_bumps}.

\begin{cor} \cite{RubinTheorem} \label{Rubin_cor}
Suppose that $G,H\leq \PLoI$ are nontrivial and each acts on its support in a locally dense manner.
If $\phi: G \to H$ is an isomorphism, then there is a unique homeomorphism
$\psi:\supt G \to \supt H$ such that
for all $x \in \supt(G)$, $\psi(x g) = \psi(x)\phi(g)$.
\end{cor}

\begin{rem}
Both McCleary \cite{McCleary} and Bieri-Strebel \cite{BieriStrebel} had previously
proved similar reconstruction theorems for subgroups of $\PLoI$,
although under different dynamical hypotheses. 
\end{rem}

\begin{proof}
First observe that since the action of $G \leq \PLoI$ on its support is locally dense,
then the only $G$-equivariant maps between orbitals of $G$ are the identity functions.
This in particular implies that $\psi$ is unique if it exists.
To prove existence, suppose $G,H \leq \PLoI$  and $\phi:G \to H$ are as in the statement of the corollary.
By replacing $G$ and $H$ by rescaled translates if necessary, we may assume that 
the supports of $G$ and $H$ are disjoint.

Define $\Gamma := \{g \phi(g) \mid g \in G\}$ and let $J$ be an orbital of $G$.
Observe that we are finished once we have shown that there is a $\Gamma$-equivariant homeomorphism
between $X$ and $Y$.
Furthermore it suffices to show that for each orbital $J$ of $G$, there a unique orbital $K$ of $H$ 
for which there is a $\Gamma$-equivarient homeomorphism between $J$ and $K$.
The statement with the roles of $G$ and $H$ reverse must also hold and $\psi$ is then obtained by
pasting together these local homeomorphisms.

Fix a $g_0 \in G$ such that $\supt(g_0)$ is nonempty with closure contained in $J$.
Let $K_i$ $(i < n)$ list the orbitals of $H$ which intersect $\supt(\phi(g_0))$ --- there are only finitely many such
orbitals since $\phi(g_0) \in \PLoI$.
Since $J$ is a resolvable orbital of $G$,
the only element of $G |_J$ which commutes with every conjugate of $g_0$ is the identity.
Apply Theorem \ref{semiconj_dichot} to the group $\Gamma$ and observe that
the first alternative cannot hold since if $g_J$ is not the identity, $g$ fails to commute with $g_0^h$ for some
$h \in G$.
Since the support of $\phi(g_0^h)$ is contained in $\bigcup_{i < n} K_i$, it must be that
$\phi(g) |_{K_i}$ is nontrivial for some $i < n$.
Thus there is an $i < n$ and a $\Gamma$-equivariant continuous surjection
$\theta :K_i \to J$.
Similarly, there is a $\Gamma$-equivariant continuous surjection $\vartheta$ from some orbital $J'$ of 
$G$ to $K_i$
By the observation made at the start of the proof, $J' = J$ and
$\vartheta = \theta^{-1}$.
In particular $\vartheta :J \to K$ is the desired $\Gamma$-equivariant homeomorphism.
\end{proof}

\begin{rem}
It should be noted that Theorem \ref{semiconj_dichot} is false if we replace $\PLoI$ with
$\HomeoI$.
For example, since $F$ is orderable \cite{simple_orderable},
there is a $G \leq \HomeoI$ which is isomorphic to $F$ such that
every nonidentity element of $G$ has only isolated fixed points;
such a $G$ can not be semiconjugate to the standard copy of $F$.
It would be interesting to know if there are broader contexts in which Theorem \ref{semiconj_dichot} holds.
\end{rem}

Next we will derive Brin's Ubiquity Theorem from Theorem \ref{semiconj_dichot}.

\begin{cor} \cite{ubiquity} \label{ubiquity_cor}
Suppose $G \leq \PLoI$ and $K$ is an orbital of $G$ such that some element of $G$
approaches one end of the $K$ but not the other.
Then $F$ embeds into $G$.
\end{cor}

\begin{proof}
By replacing $G$ with a rescaled translate if necessary, we may assume that the support of $G$
is contained in $(1/2,1)$.
Let $a$ and $b$ be the generators for the rescaled standard model of $F$ 
with support $(0,1/2)$.
Let $K_i$ $(i < n)$ list the orbitals of $G$ so that $K_0 = K$.
The hypothesis combined with Lemma \ref{lem:mov} readily yields
a pair $f,g \in G$ such that $f|_J$ and $g|_J$ 
satisfy the same relations as $a$ and $b$.

Define $\Gamma := \gen{af, bg}$ and apply Theorem \ref{semiconj_dichot}
to the group $\Gamma$, the distinguished orbital $J:=(0,1/2)$, and the orbitals
$K_i$ $(i < n)$.
There is a subset $X \subseteq \{0,\ldots,n-1\}$ and $\Gamma$-equivarient continuous monotone surjections 
$\psi_i:K_i \to J$ for $i \in X$ and an $h \in \Gamma$ such that if $i  < n$ is not in $X$,
then $h |_{K_i}$ is the identity;
let $\psi: \bigcup_{i \in X} K_i \to J$ be the common extension of the $\psi_i$'s.
Using that $J$ is a resolvable orbital of $\Gamma$ and arguing as in the proof of Proposition \ref{dense_bumps}, there is an $h_0$ in the normal closure of $h$ in $\Gamma$
such that $h_0 |_J$ is a positive bump. 
Observe that the image of the support of $h_0$ under $\psi$ is the union of $(s_0,t_0):=\supt(h_0) \cap J$ and a finite set $E$.
Let $g$ be such that $E g \cap E = \emptyset$ and $s_0 < s_0 g < t_0 < t_0 g$.
It is now easily checked that for some $m > 0$, $a:=h_0^m$ and $b:= (h_0^g)^{-m}$ are as in Proposition \ref{fastF}.
\end{proof}

\begin{rem}
While we used Brin's Ubiquity Theorem to prove Theorem \ref{F_embed}, it is not
required for the proof of Theorem \ref{semiconj_dichot}.
Even so, the purpose of deriving Corollaries \ref{Rubin_cor} and \ref{ubiquity_cor} from Theorem \ref{semiconj_dichot} is
not to give new proofs of these facts but rather to demonstrate the ways in which Theorem \ref{semiconj_dichot}
can be used and the utility that resides in it.
\end{rem}

\section{Some examples}

\label{examples:sec}

In this section, we will prove Corollaries \ref{Ftau}--\ref{trans_F}.
Recall that Cleary's group $F_\tau$ is the subgroup of $\PLoI$ consisting of those
elements whose singularities are in $\Z[\tau]$ and whose slopes are powers of $\tau$,
where $\tau$ is the solution to $\tau^2 = \tau + 1$ with $\tau > 1$.
If $1 < p < q$ are relatively prime integers, Stein's group $F_{p,q}$ is the subgroup of $\PLoI$ consisting
of those elements whose singularities are in $\Z[\frac{1}{p},\frac{1}{q}]$ and whose slopes are the product of a power of $p$ and
a power of $q$.

The following observations will allow us to show that Cleary's and Stein's groups contain
$F$-obstructions.

\begin{obs} \label{trans_obs}
Suppose that $f,g \in \HomeoI$ and for some $s$ and $0 < \xi < \eta$,
$xf = x+\xi$ and $xg = x+\eta$ whenever $s \leq x \leq s+\eta$.
Then the rotation number of $f$ modulo $g$ at $s$ is defined and equals $\xi/\eta$.
\end{obs}

\begin{obs} \label{ratio_obs}
Suppose that $f,g \in \HomeoI$ and for some $s_0 < s_1$ and  $1 < a < b$,
$xf = a(x-s_0) + s_0$ and $xg = b(x-s_0) + s_0$ whenever
$s_0 \leq x \leq s_1$.
If $s \in (s_0,s_1)$ is such that $sg \leq s_1$, then the rotation number of $f$
modulo $g$ at $s$ is defined and equals $\log_b (a)$
\end{obs}

The second observation is a consequence of the first by conjugating
$f$ and $g$ by $\log_b \frac{x - s_0}{s}$.
If $1 < p < q$ are relatively prime integers, then $\log_q(p)$ is irrational.
Since $F_{p,q}$ contains elements which have slope $p$ and $q$ near $0$, Corollary \ref{Fpq} follows
from Observation \ref{ratio_obs} and Theorem \ref{main_result}.

We now turn to Cleary's group $F_\tau$.
Define $f,g \in F_\tau$ by
$$
x f :=
\begin{cases}
x \tau & \textrm{ if } 0 \leq x \leq \tau^{-3} \\
x + \tau^{-2} - \tau^{-3} & \textrm{ if } \tau^{-3} \leq x \leq \tau^{-1} \\
x \tau^{-1} + \tau^{-2} & \textrm{ if } \tau^{-1} \leq x \leq 1
\end{cases}
$$
$$
xg :=
\begin{cases}
x \tau^2 & \textrm{ if } 0 \leq x \leq \tau^{-4} \\
x + \tau^{-2} - \tau^{-4} & \textrm{ if } \tau^{-4} \leq x \leq \tau^{-1} \\
x \tau^{-2} + \tau^{-1} & \textrm{ if } \tau^{-1} \leq x \leq 1
\end{cases}
$$
If we set $s:= \tau^{-3}$, then
$$sf = \tau^{-2} < \tau^{-2} + \tau^{-3}- \tau^{-4} = \tau^{-1} - \tau^{-4} = sg < \tau^{-1}.$$
It follows from Observation \ref{trans_obs} that the rotation number of $f$ modulo $g$ at
$s$ is defined and equals 
$$
\frac{\tau^{-2} - \tau^{-3}}{\tau^{-2} - \tau^{-4}} = \frac{\tau^2 - \tau}{\tau^2 - 1} = \tau^{-1}.
$$
Since $\tau^{-1}$ is irrational, $(f,g)$ is an $F$-obstruction.

Finally, we wish to show that the group generated by $F \cup F^{t \mapsto t - \xi}$ contains an $F$-obstruction
whenever $0 < \xi < 1$.
Recall that $F$ is the subgroup of $\PLoI$ consisting of those elements whose singularities occur at
dyadic rationals and whose slopes are powers of $2$.
Let $\xi$ be given and let $n$ be such that $2^{-n} < \xi < 1-2^{-n+2}$.
Observe that the following functions $f$, $g_0$, and $g_1$
are in either $F$ or $F^{t \mapsto t - \xi}$:
$$
x f:=
\begin{cases}
2x+\xi & \textrm{ if } -\xi \leq x \leq 2^{-n} - \xi \\
x+2^{-n} & \textrm{ if } 2^{-n}-\xi \leq x \leq 1-2^{-n+1} - \xi \\
2^{-1} (x + 1 - \xi) & \textrm{ if } 1-2^{-n+1} - \xi \leq x \leq 1-\xi \\
x & \textrm{ otherwise} 
\end{cases}
$$
$$
xg_0:=
\begin{cases}
2x & \textrm{ if } 0 \leq x \leq 2^{-1} - 2^{-n-1} \\
x + 2^{-1} - 2^{-n-1} & \textrm{ if } 2^{-1} - 2^{-n-1} \leq x \leq 2^{-1} \\
2^{-n} x + 1 - 2^{-n} & \textrm{ if } 2^{-1} \leq x \leq 1 \\
x & \textrm{ otherwise}
\end{cases}
$$
$$
xg_1:=
\begin{cases}
2^n (x+\xi) - \xi  & \textrm{ if } -\xi \leq x \leq 2^{-n-1} -\xi \\
x + 2^{-1}-2^{-n-1} & \textrm{ if } 2^{-n-1} - \xi \leq x \leq 2^{-n} -\xi \\
2^{-1} (x+\xi) + 2^{-1}- \xi & \textrm{ if } 2^{-n}- \xi \leq x \leq 1-\xi \\
x & \textrm{ otherwise}
\end{cases}
$$
Set $g := g_0 g_1$.
Observe that by our choice of $n$, if $0 \leq x \leq (1-\xi)/2$, then:
$$
x f = x + 2^{-n}
\qquad
x g = x + (1-\xi)/2.
$$
Since  $0f = 2^{-n} < (1-\xi)/2 = 0g$, it follows from Observation \ref{ratio_obs}
that the rotation number of $f$ modulo $g$ at $0$ is defined and equals
$2^{-n+1}/(1-\xi)$, which is irrational.
Hence $(f,g)$ is an $F$-obstruction.

\begin{rem}
We do not know if $\gen{\bigcup_{q \in \Q} F^{t \mapsto t -q}}$ embeds into $F$.
We conjecture it does not.
Note that if $G \leq \gen{\bigcup_{q \in \Q} F^{t \mapsto t -q}}$ is finitely generated,
then $G$ is conjugate to a subgroup of $F$.
Specifically if $X \subseteq \frac{1}{n} \Z$, then 
$\gen{\bigcup_{q \in X} F^{t \mapsto t -q}}$ is
conjugate to a subgroup of the real line model
of $F$ via the map $t \mapsto nt$.
\end{rem}


\begin{thebibliography}{10}

\bibitem{BieriStrebel}
R. Bieri, R. Strebel, \emph{On groups of {PL}-homeomorphisms of the real line},
in \emph{Mathematical Surveys and Monographs}, \textbf{215}.
American Math. Society, Providence, RI., 2016. xvii+174pp.

\bibitem{alg_class}
C. Bleak, \emph{An algebraic classification of some solvable groups of
  homeomorphisms}, J. Algebra \textbf{319} (2008), no.~4, 1368--1397.

\bibitem{fast_gen}
C. Bleak, M. G. Brin,  M. Kassabov,  J.Tatch Moore, and M. C. B. Zaremsky, 
\emph{Groups of fast homeomorphisms of the interval and the
              ping-pong argument}, J. Comb. Algebra
\textbf{44} (2019), no.~3, 1--40.

\bibitem{ubiquity}
M.~G. Brin, \emph{The ubiquity of {T}hompson's group ${F}$ in groups of
  piecewise linear homeomorphisms of the unit interval}, J. London Math. Soc.
  (2) \textbf{60} (1999), no.~2, 449--460.

\bibitem{PLoI}
M. G. Brin and C. G. Squier,
\newblock Groups of piecewise linear homeomorphisms of the real line.
\newblock {\em Invent. Math.}, 79(3):485--498, 1985.
  
\bibitem{pres_cong_PLoI}
\bysame,
\emph{Presentations, conjugacy, roots, and centralizers in groups of
piecewise linear homeomorphisms of the real line}, Comm. Algebra
\textbf{29} (2001), no.~10, 4557--4596.

\bibitem{CFP}
J.~W. Cannon, W.~J. Floyd, and W.~R. Parry, \emph{Introductory notes on
  {R}ichard {T}hompson's groups}, Enseign. Math. (2) \textbf{42} (1996),
  no.~3-4, 215--256. 

\bibitem{simple_orderable}
C.~G.~Chehata,  \emph{An algebraically simple ordered group},
Proc. London Math. Soc. (3) \textbf{2} (1952), 183--197.

\bibitem{irrational_PLoI}
S. Cleary, \emph{Groups of piecewise-linear homeomorphisms with irrational slopes}, Rocky Mountain J. Math.
\textbf{25} (1995), no.~3, 935--955.

\bibitem{Ftau}
\bysame, \emph{Regular subdivision in {$\mathbf{Z}[\frac{1+\sqrt 5}{2}]$}}, Illinois J. Math.
\textbf{3} (2000), no.~1, 453--464.


\bibitem{simple_homeo}
D. B. A. Epstein,  \emph{The simplicity of certain groups of homeomorphisms},
Compositio Math. \textbf{22} (1970), no.~2, 165--173.

\bibitem{T_rational_rot}
\'{E}. Ghys and V. Sergiescu, \emph{Sur un groupe remarquable de diff\'{e}omorphismes du cercle},
Comment. Math. Helv. \textbf{62} (1987), no.~2, 185--239.

\bibitem{KassabovMatucci}
M. Kassabov and F. Matucci,
\emph{The simultaneous conjugacy problem in groups of piecewise linear functions}, 
Groups Geom. Dyn. \textbf{6} (2012), no.~2, 279--315. 

\bibitem{coherent_actions}
Y. Lodha,  \emph{Coherent actions by homeomorphisms on the real line or an interval},
to appear in Israel J. Math.

\bibitem{MacKay}
R. S. MacKay, \emph{A simple proof of {D}enjoy's theorem},
Math. Proc. Cambridge Philos. Soc. \textbf{103} (1988), no.~2, 299--303.

\bibitem{McCleary}
S. H. McCleary, \emph{Groups of homeomorphisms with manageable automorphism groups},
Comm. Algebra \textbf{6} (1978), no.~5, 497--528.

\bibitem{RubinTheorem}
M. Rubin,  \emph{On the reconstruction of topological spaces from their groups of homeomorphisms},
Trans. Amer. Math. Soc. \textbf{312} (1989), no.~2, 487--538.

\bibitem{grps_PLoI}
M. Stein, \emph{Groups of piecewise linear homeomorphisms}, Trans. Amer. Math. Soc.
 \textbf{332} (1992), no.~2, 477--514. 
 
\end{thebibliography}
\end{document}